\documentclass[14pt]{amsart}
\usepackage{amsmath, amsthm}

\newtheoremstyle{theorem}
  {10pt}		  
  {10pt}  
  {\sl}  
  {}     
  {\bf}  
  {. }    
  { }    
  {}     
\theoremstyle{theorem}
\newtheorem{theorem}{Theorem}
\newtheorem{corollary}[theorem]{Corollary}
\newtheorem{proposition}[theorem]{Proposition}
\newtheoremstyle{defi}
  {10pt}		  
  {10pt}  
  {\rm}  
  {}     
  {\bf}  
  {. }    
  { }    
  {}     
\theoremstyle{defi}
\newtheorem{definition}[theorem]{Definition}
\newtheorem{remark}[theorem]{Remark}

\begin{document}

\title{Exponential stability of C0-Semigroup via Lyapunov inequality in Banach space}

\author{Belabbas Madani$^1$ Zohra Bendaoud$^2$
}
\maketitle
$^1$University of Oran 1 Ahmed Ben Bella, Algeria.
$^1$$^,$$^2$ENS of Laghouat, Laghouat, Algeria\\
$^1$b.madani@ens-lagh.dz  \
$^2$z.bendaoud@ens-lagh.dz\\
 
\begin{abstract}
We give a relation between the exponential stability of $ C_{0}- $semigroup $ \textbf{T}=\left\lbrace T(t) \right\rbrace_{t\geq 0}  $ and the solutions of Lyapunov inequality
\( \left\langle QAx,x\right\langle +\left\langle Qx,Ax\right\langle \leq -||x||^{2},   \)
in $ B^{+}(X,X^{*}) $, with $ X $ is a Banach space. The solutions of this inequality characterizes, the boundedness of the resolvent $ R(\lambda,A) $ inside and outside of the left half-plane $ \Re\lambda\geq 0 $, and also the left invertibility of the $ C_{0}- $semigroup T.

\end{abstract}

\section{Introduction}

A classical result of R. Datko \cite{Datko} stats that in Hilbert space H the $ C_{0}- $semigroup $T=\left\lbrace T(t) \right\rbrace_{t\geq 0}  $ generated by A is uniformly exponentially stable if and only if there exists a self-adjoint positive operator $ Q\in B(H) $ such that
\[ \left\langle QAx,x\right\rangle + \left\langle Qx,Ax\right\rangle =-||x||^{2}, \quad \forall x\in D(A).  \] 
The authors in \cite{Phat} and \cite{Preda} studied the equivalence between the solvability of Lyapunov operator equation (respectively inequality) and the exponential stability in Banach space. However, the strongly positive taken in the two proofs may not satisfy in arbitrary Banach space. So we give a necessary and sufficient condition for a Banach space isomorphic to a Hilbert space and also a sufficient inequality for Banach space which helps us to establish some new results.

\section{Preliminaries}

We denote by $ X $ a Banach space and by $ X^{*} $ its conjugate dual space i.e. the set of all bounded and antilinear functionals on $ X $, the duality pairing between $ X^{*} $ and $ X $ denote by $ \left\langle .,. \right\rangle  $ and by $ B(X,X^{*}) $ the space of all linear and bounded operators from $ X $ to $ X^{*} $. An operator $ Q\in B(X,X^{*})  $ will be called symmetric if $ \left\langle Qx,y \right\rangle = \left\langle Qy,x \right\rangle $ for all $ x,y\in X $, and is called positive if $\left\langle Qx,x \right\rangle \geq 0$ , for all $ x\in X $. If $ Q $ is positive operator with $ \left\langle Qx,x\right\rangle \geq \theta \left\| x\right\| ^{2} $, for some $ \theta >0 $ and all $ x\in X $, we called it strongly positive. Denote by $ B^{+}(X,X^{*}) $ the space of all positive operator on $ B(X,X^{*}) $.
$ \textbf{T}=\left\lbrace T(t) \right\rbrace_{t\geq 0} $ always $ C_{0}- $semigroup on $ X $ with generators $ A $. The $ C_{0}- $semigroup $ \textbf{T} $ is called uniformly exponentially stable if there exist constants $ \varepsilon >0, M\geq 1 $ such that \[ \left\| T(t) \right\|\leq Me^{-\varepsilon t},\: t\geq 0.  \]
It is well-known from \cite{Lin} (see also\cite{Dariva}) that 
\begin{theorem}\label{Lin}
	A Banach space $ X $ has the property that there exists a strongly positive operator $ P $ from $ X $ onto $ X^{*} $ such that $ \left\langle Px,x \right\rangle \geq \theta \left\|x \right\|^{2}   $ for some $ \theta>0 $, if and only if, the Banach space $ X $ is isomorphic to a Hilbert space.
	
\end{theorem}

From the proof in \cite{Lin}, the theorem is still true if the operator $ P $ it should be not assumed onto.

\section{Exponential stability }
\begin{theorem}\label{1}
	Let $ X $ be a Banach space isomorphic to a Hilbert space and $ A $ the generator of $ C_{0}-$semigroup $ \lbrace T(t)\rbrace_{t\geq 0}$. Then the following are equivalent.
	\begin{enumerate}
		\item $ \lbrace T(t)\rbrace_{t\geq 0}$ is uniformly exponentially stable, with $ ||T(t) ||\leq Me^{\omega_{0}t} \qquad (M\geq 1,  \omega_{0}<0) $.
		\item There exists an operator $ Q\in B^{+}(X,X^{*}) $ such that
		\begin{equation}\label{ineq1}
			\left\langle QAx,x\right\rangle +\left\langle Qx,Ax\right\rangle \leq -\left\|x \right\|^{2}, \: \forall x\in D(A).
		\end{equation}
		
	\end{enumerate}
	
\end{theorem}
\begin{proof}
	$ (1).\Rightarrow (2). $ By theorem \ref{Lin} there exists an operator $ P\in B(X,X^{*}) $ such that $ \left\langle Px,x\right\rangle \geq \theta\left\|x \right\|^{2}$  for some $ \theta>0 $ and all $ x\in X $. We consider the linear operator $ Q_{0}B(X,X^{*}) $ defined by 
	\[ \left\langle Q_{0}x,y\right\rangle:=\dfrac{1}{\theta}\int_{0}^{+\infty}\left\langle PT(t)x,T(t)y \right\rangle dx, \: x,y\in X.   \] 
	Since $ \textbf{T} $ is uniformly exponentially stable, $ Q_{0} $ is well-defined. We show that $ Q_{0} $ is positive bounded operator, this follows from the following, for all $ x\in X $ we have
	\[ \left\|Q_{0}x \right\|=\sup_{||y||\leq 1 } |\left\langle Q_{0}x,y\right\rangle|\leq \int_{0}^{+\infty}\sup_{||y||\leq 1 } | \left\langle PT(t)x,T(t)y \right\rangle | dt \leq \dfrac{||P||M^{2}}{-2\omega_{0}}||x||^{2},  \]
	also
	\[ \left\langle Q_{0}x,x\right\rangle=\frac{1}{\theta}\int_{0}^{+\infty}\left\langle Q_{0}x,x\right\rangle \geq \int_{0}^{+\infty}||T(t)||^{2} dt \geq 0. \]
	
	Finally, since for all $ x,y \in D(A), $
	\[ \frac{d}{dt} \left\langle PT(t)x,T(t)x\right\rangle =\left\langle PT(t)Ax,T(t)x\right\rangle +\left\langle PT(t)x,T(t)Ax\right\rangle, \]  and 
	\[\int_{0}^{+\infty}\frac{d}{dt}\left\langle PT(t)x,T(t)y\right\rangle = -\left\langle Px,y\right\rangle, \]
    therefore
	\[ \left\langle Q_{0}Ax,x\right\rangle \left\langle Q_{0}x,Ax\right\rangle =-\frac{1}{\theta}\left\langle Px,x\right\rangle\leq -||x||^{2}, \: \forall x\in D(A). \]
		
	$ 2.\Rightarrow 1. $ Since \ref{ineq1} holds, by replacing $ x $ by $ T(t)x $ we have 
	\[ \left\langle QAT(t)x,T(t)x\right\rangle + \left\langle QT(t)x,AT(t)x\right\rangle  \leq -||T(t)x||^{2}, \]
	$\forall x\in D(A), \ t\geq 0,$ this means that $ \frac{d}{dt}\left\langle QT(t)x,T(t)x\right\rangle \leq -\left\|T(t)x \right\|^{2}.   $	
	
	By integration over $ (0,s), $ we obtain for all $ x\in D(A) $ and $ s>0 $
	\[ -\left\langle QT(s)x,T(s)x\right\rangle + \left\langle Qx,x\right\rangle \geq \int_{0}^{s}\left\|T(t)x \right\|^{2} dt.  \]
	Since $ Q $ is positive bounded operator and $ A $ is a densely defined operator on $ X $, then 
	\[ \int_{0}^{s}\left\|T(t)x \right\|^{2} dt \leq \left\langle Qx,x\right\rangle \leq ||Q||||x||^{2}, \forall x\in X, s>0,   \]		applying  Datko-Pazy theorem (see Theorem 3.1.8, \cite{Neerven2}) we obtain the exponential stability of $ \textbf{T}.  $
	
\end{proof}

\begin{remark}
	From the proof above, it should be noted that the result $ 2.\Rightarrow 1. $ is still true if X is arbitrary Banach space.
\end{remark}

Let $ L_{A}$ the set of all operators $ Q\in B^{+}(X,X^{*}) $ such that, the Lyapunov inequality $ \left\langle QAx,x\right\rangle +\left\langle Qx,Ax\right\rangle \leq -\left\|x \right\|^{2}  $ hold for all $ \forall x\in D(A) $. With this notation, $ L_{A}$ is non-empty implies the exponential stability of $ \textbf{T} $ in arbitrary Banach space.
It is clear that if $ Q\in  L_{A} $ then $ cQ\in  L_{A} $ for all $ c>0 $. We can show easily that $ L_{A}$ is closed subset of $ B^{+}(X,X^{*}) $, using that if $ \left\lbrace Q_{n}\right\rbrace_{n}\subset L_{A}  $ converging to $ Q $  then for $ x\in D(A) $
\[ \left\langle QAx,x\right\rangle+\left\langle Qx,Ax\right\rangle =\left\langle (Q-Q_{n})Ax,x\right\rangle +\left\langle (Q-Q_{n})x,Ax\right\rangle +\left\langle Q_{n}Ax,x\right\rangle +\left\langle Q_{n}x,Ax\right\rangle. \]

With \[ \left\| \left\langle (Q-Q_{n})Ax,x\right\rangle +\left\langle (Q-Q_{n})x,Ax\right\rangle \right\| \leq 2\left\|Q-Q_{n} \right\|\left\| Ax\right\|\left\| x\right\|     \]
and by letting $ n\rightarrow +\infty $  we obtain $ Q\in D(A) $.
\begin{corollary}
	Let $ A $ be the generator of $ C_{0}- $semigroup $ \textbf{T} $ on Banach space isomorphic to a Hilbert space. If $ \textbf{T} $ is uniformly exponentially stable and $ B\in B(X) $ such that $ \left\| B \right\|\leq \dfrac{1}{2\alpha || Q|| }$, for some $ Q\in L_{A} $ and $ \alpha>1 $. Then the $ C_{0}-$semigroup generated by $ A+B $ is uniformly exponentially stable with $ Q\in L_{A+B} $.
\end{corollary}

\begin{proof}
If $ B\in B(X) $ such that $ \left\| B \right\|\leq \dfrac{1}{2\alpha || Q|| }$ for some $ Q\in L_{A} $ and $ \alpha >1 $, then for all $ x\in D(A) $ we have

\[  \begin{array}{ll}
	\left\langle Q(A+B)x,x\right\rangle +\left\langle Qx,(A+B)x\right\rangle & \leq -\left\| x\right\|^{2}+ \left\langle QBx,x\right\rangle+\left\langle Qx,Bx\right\rangle     \\ 
	 & \leq -\left\| x\right\|^{2}+2\left\| Q\right\|\left\| B\right\|\left\| x\right\|^{2} \\
	 &\leq -(\frac{\alpha -1}{\alpha})\left\| x\right\|^{2} .\end{array}, \] 
then we obtain the desired result.	
\end{proof}
In the following we establish a result of uniformly boundedness of the resolvent inside and outside the left half-plane $ Re \lambda\geq 0 $, for this we need some quantities; with  $\sigma(A)$ the spectrum set of $ A $, the spectral bound of A, is defined by $$ s(A)=\sup \left\lbrace Re \lambda\in \mathbb{R}: \lambda\in\sigma(A) \right\rbrace  $$,
the growth bound, is defined by 
$$ \omega_{0}=\inf\left\lbrace \omega\in \mathbb{R} : \exists M\geq 1 \ such \ that \left\| T(t)\right\| \leq Me^{\omega t} for \ all \ t\geq 0 \right\rbrace  $$.

If $ \textbf{T} $ is uniformly exponentially stable, then $ s_{0}(A)\leq s_{0}(A)\leq \omega_{0} $ (see (\cite{Engel}, Proposition 2.2) and \cite{Neerven2}).

\begin{theorem}
Let $ A $ be the generator of $ C_{0}- $semigroup $ \textbf{T} $ on Banach space $ X $ suth that $ L_{A} $ non-empty. Then a boundedness of the resolvent of A is given by
\begin{equation}\label{bound1}
 \sup_{Re \lambda\geq 0 }\left\|R(\lambda, A) \right\|\leq 2\inf_{Q\in L_{A}}\left\|Q \right\|,
\end{equation}

and there exists $ \delta_{0}>0 $ such that
\begin{equation}\label{bound2}
	\sup_{Re \lambda\leq -\delta_{0} }\left\|R(\lambda, A) \right\|\leq \dfrac{2\left\| Q\right\| }{1-2\delta_{0}\left\|Q \right\|}, \quad   \forall Q\in L_{A}.
\end{equation}
\end{theorem}

\begin{proof}
	For the resolvent set $ \rho(A) $, we recall from (\cite{Neerven2}, Corollary 2.2.2.) that, $ \textbf{T} $ is uniformly exponentially stable then $ \left\lbrace Re\lambda\geq 0 \right\rbrace \subset \rho (A)  $. Now let $ \lambda\in\rho(A) $ and $ Q\in L_{A} $ then we have
	\[ \left\langle Q(A-\lambda)x,x \right\rangle +\left\langle Qx,(A-\lambda)x \right\rangle \leq -\left\| x\right\|^{2}-2Re\lambda \left\langle Qx,x\right\rangle, \quad \forall x\in D(A).     \]
	
	Let $ y=(\lambda -A)x $, then the above inequality becomes
	\[ \left\langle Qy,R(\lambda,A)y\right\rangle +\left\langle QR(\lambda,A)y,y\right\rangle\leq-\left\|R(\lambda,A)y \right\|^{2}-2Re\lambda\left\langle QR(\lambda,A)y,R(\lambda,A)y \right\rangle, \quad \forall y\in X.     \]
	
	If $ Re\lambda\geq 0 $, then
	\[ -\left\langle Qy,R(\lambda,A)y\right\rangle -\left\langle QR(\lambda,A)y,y \right\rangle \leq -\left\|R(\lambda,A)y \right\|^{2}, \quad \forall y\in X,    \]
	
	consequently, fot all $ y\in X $ and all $ Q\in L_{A} $,
	\[ \left\|R(\lambda,A)y \right\|^{2}\leq \left\langle Qy, R(\lambda,A)y \right\rangle + \left\langle QR(\lambda,A)y,y \right\rangle \leq 2\left\| Q \right\|\left\| R(\lambda,A)y \right\| \left\| y\right\|,       \]
		
	thus
	\[ \sup_{Re \lambda\geq 0 }\left\| R(\lambda,A)\right\|\leq 2\inf _{Q\in L_{A}} \left\| Q\right\|.\]
	Since the resolvent is uniformly bounded in the open right half plane then $ s_{0}(A)<0 $, moreover, for $ s(A)<Re\lambda<0 $ we obtain that
	\[ \left\langle Qy,R(\lambda,A)y \right\rangle +\left\langle QR(\lambda,A)y,y \right\rangle \geq (1+2\left\| Q\right\| Re\lambda)\left\| R(\lambda,A)y\right\|^{2},    \]
	then
	\[ 2\left\|Q \right\| \left\| R(\lambda,A)y\right\| \left\| y\right\| \geq (1+2\left\| Q\right\| Re\lambda)\left\| R(\lambda,A)y\right\| ^{2},  \]
	
	and therefore
	\[ \left\| R(\lambda,A)y\right\|\leq \dfrac{2\left\| Q\right\| }{1+2\left\| Q\right\| Re\lambda}\left\| y\right\|, \quad \forall y\in X.   \]
	
	Since $ \alpha Q\in L_{A} $ for all $ \alpha \geq 1 $, we can choose $ \alpha \geq 1 $ such that $ \dfrac{-1}{2\alpha \left\|Q \right\| }\geq s(A) $. With $ \delta_{0} $ fixed positive real numbre and $ Q\in L_{A} $ such that $ -\delta_{0}>\dfrac{-1}{2\alpha \left\| Q\right\| } $, we have \[ \dfrac{2\left\| Q\right\| }{1+2\left\| Q\right\| Re\lambda}\leq \dfrac{2\left\| Q\right\| }{1-2\delta_{0}\left\| Q\right\| }, \]
		thus 
\[ \sup_{Re \lambda\leq -\delta_{0} }\left\|R(\lambda, A) \right\|\leq \dfrac{2\left\| Q\right\| }{1+2\delta_{0}\left\|Q \right\|}, \quad   \forall Q\in L_{A}, \]	
	with $ \dfrac{2\left\| Q\right\| }{1-2\delta_{0}\left\|Q \right\|} $ depend only on the value of $ \left\| Q\right\| $.
\end{proof}
\section{Left invertibility}

In following we study the relation between Lyapunov inequality and the left invertibility of $ \textbf{T} $. We begin with the definition of left invertibility.
    
\begin{definition}
The $ C_{0}- $semigroup $ \textbf{T} $ is left-invertible if there exists a function $ t\mapsto m(t) $ such that $ m(t)>0 $ and for all $ x\in X $ there holds 
\[ \left\| T(t)x\right\|\geq m(t)\left\| x\right\|, \quad  \forall t\geq 0.   \]
\end{definition}

We need the following results from \cite{Xu} (see also \cite{Zwart}):

\begin{proposition}\label{left-invert1}
Let $ \textbf{T} $ be a $ C_{0}- $semigroup on Banach space $ X $. Then the following statements are equivalent. 
\begin{enumerate}
	\item $ \textbf{T} $ is left-invertible.
	\item There exists $ t_{0}>0$  and $ m_{0}>0 $, such that $ \left\| T(t_{0})\right\| \geq m_{0}\left\| x\right\|$,   for all $x\in X$. 
\end{enumerate}
\end{proposition}

\begin{proposition}\label{left-invrt2}
	Let $ \textbf{T} $ be a $ C_{0}- $semigroup on Banach space $ X $. Then the following statements are equivalent.
	\begin{enumerate}
		\item $ \textbf{T} $ is left-invertible.
		\item There exist two constants  and $\alpha, c>0 $ such that \[ \left\| T(t)x\right\|\geq ce^{-\alpha t}\left\| x\right\| \quad \forall x\in X, \ t\geq 0.  \]
				
	\end{enumerate}
\end{proposition}

\begin{theorem}
	Let $ A $ be the generator of $ C_{0}- $semigroup $ \textbf{T} $ on Banach space $ X $ isomorphic to a Hilbert space. If $ \textbf{T} $ is uniformly exponentially stable, then the following are equivalent.
	\begin{enumerate}
		\item $ \textbf{T} $ is left-invertible. 
		\item All of the elements in $ L_{A} $ are strongly positive.
	\end{enumerate}
\end{theorem}
\begin{proof}
	$ 1.\Rightarrow 2.$  Let $ Q\in L_{A}$, so that 
	$\left\langle QAx,x\right\rangle + \left\langle Qx,Ax\right\rangle \leq -\left\| x\right\|^{2}$ for all $ x\in D(A) $.  
	 This implies that
\[ \dfrac{d}{dt}\left\langle QT(t)x,T(t)x\right\rangle\leq -\left\| T(t)x\right\| ^{2}, \quad \forall x\in D(A).  \]	
	
	As in the proof of theorem \ref{1} use that $ A $ is a densely defined operator on $ X $, we obtain that
\[ \left\langle Qx,x\right\rangle\geq \int_{0}^{+\infty}\left\| T(t)x\right\|^{2} dt \quad \forall x\in X. \]	
	
	Thus, by the left invertibility of $ \textbf{T} $, and Proposition \ref{left-invrt2}, there exists $ \alpha, c>0 $ such that 
\[ \left\langle Qx,x\right\rangle\geq c^{2}\int_{0}^{+\infty} e^{-2\alpha s}\left\| x\right\| ^{2} ds =\dfrac{c^{2}}{2\alpha}\left\| x\right\|^{2}, \ \forall x\in X \]	

$ 2. \Rightarrow 1. $ 	
We show in the proof of theorem \ref{1} that the operator $ Q_{0}\in B(X,X^{*})$ defined by
\[ \left\langle Q_{0}x,y \right\rangle := \dfrac{1}{\theta}\int_{0}^{+\infty}\left\langle PT(t)x,T(t)y\right\rangle dt, \quad x,y\in X,  \]

is a solution of the Lyapunov inequality (\ref{ineq1}) i.e. $ Q_{0}\in L_{A} $. If $ Q $ is strongly positive then for all $ x\in X $,
\[ \dfrac{1}{\theta}\int_{0}^{+\infty}\left\langle PT(t)x,T(t)x\right\rangle ds \geq \theta \left\| x\right\| ^{2}, \]
for some $ \theta >0 $.

Since $ \dfrac{1}{\theta}\int_{0}^{+\infty}\left\langle PT(t)x,T(t)x\right\rangle ds\leq \left\| P\right\| \int_{0}^{+\infty}\left\| T(t)x\right\|^{2} dt $ and 
$$\dfrac{\theta}{\left\| P\right\| }\left\| x\right\|^{2}= \int_{0}^{+\infty}e^{-\frac{|| P||}{\theta}s }\left\| x\right\| ^{2} ds,   $$
then for all $ x\in X $,
\[ \int_{0}^{+\infty}\left\| T(s)x\right\| ^{2} ds\geq \int_{0}^{+\infty}e^{-\frac{|| P||}{\theta}s }\left\| x\right\| ^{2} ds. \]

Clealy, for two positive real functions $f(t):=\left\|T(t)x \right\|^{2}   $ and $ g(t) :=e^{-\frac{|| P||}{\theta}s }\left\| x\right\| ^{2} $, such that $ \int_{0}^{+\infty}f(t) dt \geq \int_{0}^{+\infty} g(t) dt $, there exists a set $ \Omega_{0}\subset \mathbb{R}_{+} $ with non-zero measure such that $ f(f)\geq g(t) $ for all $ t\in \Omega_{0} $, then there exists $ t_{0}\in \Omega_{0} $,  
\[ \left\| T(t_{0})x\right\| \geq e^{-\frac{|| P||}{2\theta}t_{0} }\left\| x\right\|, \]  for all $ x\in X $. By the proposition \ref{left-invert1}, $ \textbf{T} $ is left-invertible.
	
\end{proof}


\begin{thebibliography}{99}



\bibitem{Engel} K. J. Engel and R. Nagel. One-Parameter Semigroups for Linear Evolution Equations, volume 194 of Graduate Texts in Mathematics. Springer Verlag, Heidelberg, Berlin, New-York, 1999.
\bibitem{Dariva} D.Drivaliaris, N. Yannakakis.: Hilbert space structure and positive operators. J. Math. Anal. Appl. 305(2), 560--565 (2005)
\bibitem{Datko} R. Datko, Extending a theorem of Lyapunov to Hilbert spaces, J. Math. Anal. Appl. 32 (1970) 610--616.
\bibitem{Lin} B.L. Lin, On Banach spaces isomorphic to its conjugate, in: Studies and Essays (presented to Yu-why Chen on his 60th birthday, April 1, 1970), Math. Res. Center, Nat. Taiwan Univ., Taipei, 1970, pp. 151--156.
\bibitem{Preda} C. Preda and P. Preda, Lyapunov operator inequalities for exponential stability of Banach space semigroups of operators, Appl. Math. Letters 25(3) (2012), 401--403.
\bibitem{Phat} Vu Ngoc Phat, Tran Tin Kiet, On the Lyapunov equation in Banach spaces and applications to control problems, Int. J. Math. Math. Sci. 39 (3) (2002) 155--166.
\bibitem{Neerven2} J.M.A.M. van Neerven,, The asymptotic behaviour of semigroups of linear operators, Operator Theory: Advances and Applications, vol. 88, Birkhäuser Verlag, Basel, 1996. MR 1409370.
\bibitem{Xu} G.Q. Xu and Y.F. Shang, Characteristic of left invertible semigroups and admissibility of observation operators, Systems and Control Letters, 58:561--566 (2009).
\bibitem{Zwart} H. Zwart, Left-invertible semigroups on Hilbert spaces, J. Evol. Equ. 13 (2013), 335--342.



\end{thebibliography}
\end{document}